\documentclass[a4paper, 11pt]{amsart}

\usepackage[T1]{fontenc}
\usepackage[utf8]{inputenc}
\usepackage{geometry}
\usepackage{amssymb}

\makeatletter
\@namedef{subjclassname@2020}{%
  \textup{2020} Mathematics Subject Classification}
\makeatother

\newcommand{\ov}[1]{\overline{#1}}

\allowdisplaybreaks[2]

\usepackage{tikz}
\usetikzlibrary{matrix,arrows}

\usepackage[all]{xy}

\usepackage{marvosym}

\newtheorem{theorem}{Theorem}[section]
\newtheorem{proposition}[theorem]{Proposition}

\theoremstyle{definition}

\numberwithin{equation}{section}

\newcommand\NN{\mathbb{N}}
\newcommand\CC{\mathbb{C}}

\newtheorem{XxmpX}[theorem]{Example} 
\newenvironment{example1}    
  {%
   \pushQED{\qed}\begin{XxmpX}}
  {\popQED\end{XxmpX}}

  \newtheorem{XxmpY}[theorem]{Remark} 
\newenvironment{remark1}    
  {%
   \pushQED{\qed}\begin{XxmpY}}
  {\popQED\end{XxmpY}}

\begin{document}

\title[Jordan homomorphisms: A survey]{Jordan homomorphisms: A survey}

\author{Matej Brešar} 
\author{Efim Zelmanov}
\address{Faculty of Mathematics and Physics, University of Ljubljana \&
Faculty of Natural Sciences and Mathematics, University of Maribor \& IMFM, Ljubljana, Slovenia}
\email{matej.bresar@fmf.uni-lj.si}
\address{SICM, Southern University of Science and Technology,
Shenzhen, China}
\email{efim.zelmanov@gmail.com}

\thanks{The first author was partially supported by the ARIS Grants P1-0288 and J1-60025.  
The second author was partially supported by the NSFC Grant 12350710787 and the Guangdong Program 2023JC10X085.
}


\begin{abstract}The paper surveys the history and state-of-the-art   of the study of Jordan homomorphisms.
\end{abstract}

\dedicatory{Dedicated to Vesselin Drensky on his 75th birthday}


\maketitle
\section{Introduction}

 Unless stated otherwise, we assume throughout
 the paper that our algebras are defined over a field $F$ with char$(F)\ne 2$.
Let $A$ and $B$ be associative algebras.
A linear map
$\varphi:A\to B$ is called a {\bf Jordan homomorphism} if
$$\varphi(x^2)=\varphi(x)^2
 $$
for every $x\in A$.

Obvious examples of Jordan homomorphisms  are homomorphisms and antihomomorphisms. The latter are linear maps $\varphi$ satisfying $\varphi(xy)=\varphi(y)\varphi(x)$
for all $x,y\in A$.
The classical problem, around which this  paper is centered, asks whether Jordan homomorphisms can be expressed by homomorphisms and antihomomorphisms.

The definition of a Jordan homomorphism  still makes sense 
if $\varphi$ is defined not on the whole associative algebra $A$, but on a linear subspace $J$ of $A$ 
with the property that $x^2\in J$ whenever $x\in J$. Such a space is called a  {\bf special Jordan algebra}. Some  more information on these as well as  general Jordan algebras will be provided in the next section.

The paper is organized as follows.
Sections \ref{s2} and \ref{s3} explain the motivation for the study of Jordan homomorphisms, Section \ref{s4} explores some of their elementary properties, Sections \ref{s5} and \ref{s6} survey the early development of their study, and Section \ref{s7} presents some  nonstandard examples.
A more recent development is discussed in the next three sections: 
Section  \ref{s8} considers the action of Jordan homomorphisms on commutator ideals, Section \ref{s9} considers the approach by means of tetrad-eating T-ideals, and Section \ref{s10} considers the approach by means of functional identities.
The last two sections briefly discuss two  topics related to Jordan homomorphisms: Jordan derivations in Section \ref{s11} and Lie homomorphisms in Section \ref{s12}.

We will thus consider different aspects of the study of Jordan homomorphisms. However, we have no intention to be encyclopedic.  Many topics, like for example
Jordan homomorphisms on  triangular (and related) algebras and Jordan  homomorphisms appearing in Functional Analysis and some other fields of mathematics,
will be omitted or only touched upon.

Since Jordan homomorphisms  emerge in various mathematical areas,  it is  our aim to make the paper accessible to a broader audience. This is the main reason for  (mostly) restricting   to algebras over fields of characteristic different from $2$. 
Many of the results that will be stated also hold in characteristic $2$ (the definition of a Jordan homomorphism $\varphi$ in this case  requires the additional condition $\varphi(xyx)=\varphi(x)\varphi(y)\varphi(x)$, see Remark \ref{rem1}) as well as for Jordan homomorphisms of rings (which are, of course, defined as additive rather than linear maps). We ask the interested reader to consult the 
 original sources in case of doubt.

\section{First motivation: Jordan homomorphisms are constituent parts of the Jordan algebra theory}\label{s2}

A nonassociative algebra $J$ over $F$ is called a {\bf Jordan algebra} if it satisfies 
the identities
$$xy=yx\quad\mbox{and}\quad (x^2y)x=x^2(yx)$$
for all $x,y\in J$. These algebras
were introduced in 1933 by the German physicist Pascual Jordan,  in an attempt to find an algebraic setting for quantum mechanics.

\begin{example1}\label{ex1} Every associative algebra $A$ becomes a Jordan algebra if we replace the product $xy$ in $A$ by the {\bf Jordan product} 
$$x\circ y =\frac{xy+yx}{2}.$$
This Jordan algebra is denoted by $A^{(+)}$. 
\end{example1}

From
\begin{equation}\label{jsq}
    x\circ y= \frac{1}{2}((x+y)^2-x^2-y^2)
\end{equation} we see that
the condition that a linear subspace $J$ of an associative algebra $A$ is closed under squares is equivalent to the condition that $J$ is closed under the Jordan product. 
We can therefore equivalently define special Jordan algebras as follows:
A Jordan algebra  $J$ 
is said to be {\bf special} if there exists an associative algebra $A$
such that $J$ is a subalgebra of the Jordan algebra $A^{(+)}$.
\begin{example1}\label{ex2}
    Recall that an {\bf involution} on an associative algebra $A$
    is a linear antiautomorphism $*:A\to A$ that satisfies
    $(x^*)^*=x$ for every $x\in A$.
   The set of all {\bf symmetric} (also called hermitian) elements
   $$H(A,\,*\,)=\{x\in A\,|\, x^*=x\}$$
   is clearly a special Jordan algebra. 
   
   A simple concrete 
   example of an involution 
    is the transposition $a\mapsto a^t$  on the matrix algebra $M_n(F)$. The corresponding special Jordan algebra $H(M_n(F)
    ,\,t\,)$ then consists
    of all symmetric matrices.  
\end{example1}

\begin{example1} \label{ex3}   Let
$V$ be a vector space over $F$ and let
$f:V\times V\to F$ be a symmetric bilinear form. One easily checks that the vector space $J(V,f)=F\oplus V$ becomes a  Jordan algebra under the product
    $$(\lambda + v)(\mu + w)= (\lambda\mu + f(u,v)) + (\lambda w + \mu v).$$
    It is not entirely obvious, but still easy to see that this Jordan algebra is also special.
\end{example1}

A Jordan algebra that is not special is called {\bf exceptional}.

\begin{example1} \label{ex4}
    The fundamental, and under suitable assumptions also the only examples of exceptional Jordan algebras are {\bf Albert algebras}. They are  of dimension 27 over $F$. The most classical example is 
    the real vector space
 of
    $3\times 3$ self-adjoint matrices over the octonions, endowed with the  (Jordan) product
$ x\circ y={\frac {1}{2}}(x\cdot y+y\cdot x)$, where $\cdot$ stands for the usual matrix multiplication.
\end{example1}

The latter Albert algebra appeared in the  1934 Jordan-von Neummann-Wigner  classification \cite{JvNW}  of finite-dimensional Jordan algebras over $\mathbb R$ that are formally real, meaning that a sum of squares of their nonzero elements is always nonzero. It was Albert who proved just a little bit later that this algebra is exceptional \cite{Alb}.

In \cite{Z0, Z}, the second author proved that  Examples \ref{ex1}-\ref{ex4}  essentially cover all  simple (and even all prime nondegenerate) Jordan algebras.

 The motivation for the study of Jordan homomorphisms is now clear: Most Jordan algebras are special, and 
Jordan homomorphisms are nothing but homomorphisms of special Jordan algebras.

The  structure theory also indicates  that it is reasonable to restriction attention to Jordan homomorphisms on
 special Jordan algebras from Examples \ref{ex1}-\ref{ex3}.
  However, the one from  from Example \ref{ex3},  $J(V,f)$, has the property
that each $x\in J(V,f)$ satisfies
$x^2=\alpha  + \beta x$ for some
$\alpha,\beta\in F$. This makes the study of Jordan homomorphisms on 
$J(V,f)$ 
rather
special and less challenging. It is therefore  
common to restriction attention to 
 Jordan homomorphisms on 
$A^{(+)}$ and $H(A,\,*\,)$. As a matter of fact, sometimes it is enough to consider only
$H(A,\,*\,)$, for the following reason.

\begin{remark1}\label{raa} Let $A$ be an associative algebra. Its {\bf opposite algebra}
 $A^{\rm op}=\{x^{\rm op}\,|\,x\in A\}$ has the same linear structure as $A$, but the product  is defined by  $x^{\rm op} y^{\rm op} =(yx)^{\rm op}$. The  direct sum
 $A\oplus A^{\rm op}$ of algebras 
 $A$ and $A^{\rm op}$  has a natural involution, called the {\bf exchange involution},  given by $(x+y^{\rm op})^{\rm ex}=y + x^{\rm op}$.
    The Jordan algebras $A^{(+)}$ and
 $H(A\oplus A^{\rm op},\,{\rm ex}\,)$ are isomorphic, with  the isomorphism given by  $x\mapsto x+ x^{\rm op}$.\end{remark1}

Results on 
 Jordan homomorphisms on Jordan algebras of the type  $H(A,\,*\,)$
 therefore often imply similar results for Jordan algebras of the type 
 $A^{(+)}$. 

We conclude this brief overview of the theory of Jordan algebras by referring the reader to standard treaties 
\cite{Jacobson, McC, fourauthors}. 

\section{Second motivation: Jordan homomorphisms appear  across mathematics}\label{s3}

 Besides being a part of the Jordan algebra theory,  Jordan homomorphisms also live 
  their own independent life.
A quick search through the mathematical literature shows that
they are  encountered in different areas.

 Why do Jordan homomorphisms between associative algebras naturally appear? If we were asked the same question for  ordinary algebra homomorphisms, we might answer that they preserve the algebra structure and we need them to relate different algebras. A similar answer may be given for antihomomorphisms, as they preserve left-right symmetric structural  features. 
The latter is also true for
maps obtained by combining homomorphisms and antihomomorphisms, like 
maps of the form 
$\varphi:A_1\oplus A_2\to B_1\oplus B_2$, $\varphi(x_1+x_2)=\varphi_1(x_1)+ \varphi_2(x_2)$,
where $\varphi_1:A_1\to B_1$ is a homomorphism and $\varphi_2:A_2\to B_2$ is an antihomomorphism. In general, such maps are neither 
homomorphisms nor antihomomorphisms. However, they are Jordan homomorphisms. This sheds some light on  answering
our question.

A mathematical area in which Jordan homomorphisms   most naturally appear is the theory of {\bf linear preserver problems}. The aim of this  theory is to describe the form of    linear maps between algebras that preserve some properties, or some  subsets, or some relations, etc. This is a vast area  which is difficult to overview. Its roots are in Linear Algebra and Operator Theory, but over the years it has spread to other fields. We refer the reader to a few survey papers  \cite{BM, BS, LiPi, Mb, Mol}.

A very common conclusion of various linear preserver problems is that the map in question is (close to) a Jordan homomorphism.
We illustrate this by four classical results. For simplicity of exposition, we will add the assumption that the maps are unital (i.e., they send unity to unity).

We start with Frobenius' theorem from 1897 \cite{Fr}.

\begin{theorem}
    A unital linear map $\varphi:M_n(\CC)\to M_n(\CC)$
    is determinant preserving (i.e.,
    $\det (\varphi(a))=\det(a)$ for every $a\in M_n(\CC)$) if and only if $\varphi$ is a Jordan automorphism.
\end{theorem}

Frobenius, 
of course,  did not mention Jordan automorphisms in his paper (Pascual Jordan was not yet even born at that time). However, since the only Jordan automorphisms of $M_n(\CC)$  are automorphisms and antiautomorphisms (see, e.g., Theorem \ref{ThH} below), our formulation is equivalent to his.

In 1949, Dieudonné  \cite{Di} extended Frobenius' theorem as follows.
 
\begin{theorem}
    A bijective unital linear map $\varphi:M_n(F)\to M_n(F)$
    is singularity preserving (i.e., if $a$ is singular then so is $\varphi(a)$) if and only if $\varphi$ is a Jordan automorphism.
\end{theorem}

These two theorems were generalized in many different ways and are still sources of inspiration.  There has also been a parallel development in the theory of Banach algebras. In his influential booklet \cite{Kap2}, Kaplansky asked under what conditions is a unital invertibility preserving linear map between Banach algebras  a Jordan homomorphism. Various  results regarding this general question
eventually led to the following not yet solved problem: {\em  Is a bijective, unital, spectrum preserving  linear map between semisimple Banach algebras a Jordan isomorphism?}

We continue with a result that follows easily from the work of  Hua \cite{Huaa}  from 1951.

\begin{theorem}
    A bijective unital linear map $\varphi:M_n(F)\to M_n(F)$
    is rank one preserving (i.e., if $a$ has rank one then so does $\varphi(a)$) if and only if $\varphi$ is a Jordan automorphism.
\end{theorem}

In fact, Hua obtained much deeper results and his work opened a line of investigation that culminated in the work of Šemrl \cite{Peter}.

Also in 1951, Kadison \cite{K} proved the following beautiful result that provides an equivalence between  a geometric and an algebraic property of  $C^*$-algebras.  We need  a definition first: A {\bf  Jordan $\ast$-isomorphism} between $C^*$-algebras is Jordan isomorphism that also preserves selfadjoint elements.

\begin{theorem}
Let $A$ and $B$ be 
unital $C^*$-algebras.
A surjective unital linear map $\varphi:A\to B$  is an isometry if and only if $\varphi$ is a Jordan $\ast$-isomorphism. 
\end{theorem}

This is a very important result in the $C^*$-algebra theory that stimulated a  great deal of further research.

The discussion in this section
indicates that Jordan homomorphisms on Jordan algebras of the type $A^{(+)}$ are of significant interest to different groups of mathematicians. Indeed the type 
$H(A,\,*\,)$ is  more general in the sense of Remark \ref{raa}, but sometimes the type $A^{(+)}$ can be handled by methods that are not applicable to the type $H(A,\,*\,)$. We will therefore consider both types.

\section{Elementary observations}\label{s4}

Let  $F\langle X\rangle$ be
 the free associative algebra
  on the set of generators $X=\{x_1,x_2,\dots\}$. Consider
 the (special) Jordan subalgebra
$SJ\langle X\rangle$ of 
the Jordan algebra $F\langle X\rangle^{(+)}$ 
  generated by $X$. The elements of $SJ\langle X\rangle$ are called {\bf   Jordan polynomials}. 
  The simplest examples  are 
  \begin{enumerate}
      \item[(a)] $x_1x_2+x_2x_1$,
      \item[(b)] $x_1^n,\,\, n\in\NN$,
      \item[(c)] $x_1x_2x_1 = 2x_1\circ (x_1\circ x_2) - x_1^2\circ x_2$, \,\,\, and
      \item[(d)]$x_1x_2x_3+x_3x_2x_1 = (x_1+x_3)x_2(x_1 + x_3) - x_1x_2x_1 - x_3x_2x_3.$
  \end{enumerate}
  
Write $$[x,y]=xy-yx$$ for the 
  (additive) {\bf commutator}
  (or the {\bf Lie product}) of 
the algebra elements $x$ and $y$. While $[x_1,x_2]$ obviously is not a Jordan polynomial,
\begin{enumerate}\item[(e)] $[[x_1,x_2],x_3] = 4(x_1\circ (x_2\circ x_3) - x_2\circ  (x_1\circ x_3))$\,\,\, and
\item[(f)] $[x_1,x_2]^2 = 2x_1\circ x_2x_1x_2- x_1x_2^2x_1 - x_2x_1^2x_2$ \end{enumerate}
are.  

In view of \eqref{jsq},
the following proposition is now immediate.

\begin{proposition}\label{l1}
Every Jordan homomorphism $\varphi$ defined on a special Jordan algebra $J$
satisfies
\begin{enumerate}
\item[{\rm (a)}]$\varphi(xy+yx) =
\varphi(x)\varphi(y) + \varphi(y)\varphi(x)$,\item[{\rm (b)}]$\varphi(x^n) = \varphi(x)^n,\,\, n\in\NN$,\item[{\rm (c)}]$\varphi(xyx)=\varphi(x)\varphi(y)\varphi(x)$,\item[{\rm (d)}]$\varphi(xyz+zyx)=\varphi(x)\varphi(y)\varphi(z)+ \varphi(z)\varphi(y)\varphi(x)$,\item[{\rm (e)}]$\varphi(
[[x,y],z])=[[\varphi(x),\varphi(y)],\varphi(z)]$,\item[{\rm (f)}]$\varphi([x,y]^2)= [\varphi(x),\varphi(y)]^2$
\end{enumerate}
for all $x,y,z\in J$.     
\end{proposition}

\begin{remark1}\label{rem1}
Since we are assuming that char$(F)\ne 2$, (a) is  equivalent to $\varphi$ being a Jordan homomorphism (just take $x=y$). Assume for a moment that
 char$(F)=2$. Then we define a special 
Jordan algebra  as a linear subspace $J$ of an associative algebra $A$ such that
$x^2\in J$ and $xyx\in J$ whenever 
$x,y\in J$. Similarly, a Jordan homomorphism $\varphi$ defined on $J$ must satisfy two conditions:
$\varphi(x^2)=\varphi(x)^2$
and $\varphi(xyx)=\varphi(x)\varphi(y)\varphi(x)$ for all $x,y\in J$.  
\end{remark1}

Let $S$ be a subset of an associative algebra $B$. 
We write $\langle S\rangle$ 
for the subalgebra of $B$ generated by $S$.

 \begin{proposition}\label{l2}
Let $\varphi:A\to B$ be a Jordan homomorphism between associative algebras. \begin{enumerate}
    \item[{\rm (a)}]
If $e$ is a central idempotent in $A$, then $\varphi(e)$ is a central idempotent in $\langle\varphi(A)\rangle$ and
$\varphi(ex)=\varphi(e)\varphi(x)$ for every $x\in A$.
    \item[{\rm (b)}] Suppose $A$ has a unity $1$. Then $\varphi(1)$ is a unity of $\langle\varphi(A)\rangle$ and  
    if $a\in A$ is invertible, then
    so  is $\varphi(a)$  and $\varphi(a)^{-1}=\varphi(a^{-1})$. 
\end{enumerate}
\end{proposition}

\begin{proof} (a) It is clear that 
$\varphi(e)$ is an idempotent.
    Proposition \ref{l1}\,(e) shows that $$[[\varphi(x),\varphi(e)],\varphi(e)] =0$$ for every $x\in A$.
    That is, $$\varphi(x)\varphi(e)+\varphi(e)\varphi(x)=2\varphi(e)\varphi(x)\varphi(e).$$ Multiplying from the left by $\varphi(e)$ gives $\varphi(e)\varphi(x)=\varphi(e)\varphi(x)\varphi(e)$, and multiplying from the right by $\varphi(e)$ gives $\varphi(x)\varphi(e)=\varphi(e)\varphi(x)\varphi(e)$. Hence,
    $\varphi(x)\varphi(e)=\varphi(x)\varphi(e)$, so $\varphi(e)$
 lies in the center of $\langle\varphi(A)\rangle$.
 Since $ab=ba$ implies $a\circ b =ab$, the second assertion follows. 

 (b) The first assertion follows from (a). Let $a\in A$ be invertible.  By Proposition \ref{l1}, $$\varphi(a)\varphi(a^{-1})+\varphi(a^{-1})\varphi(a)=2\varphi(1)\quad\mbox{and}\quad
\varphi(a)\varphi(a^{-1})\varphi(a)=\varphi(a).$$
Hence, $e_1=\varphi(a)\varphi(a^{-1})$, $e_2=\varphi(a^{-1})\varphi(a)$
are idempotents and $e_1+e_2=2\varphi(1)$. Multiplying this equation from the left by $e_1$ gives $e_1e_2=e_1$. Similarly, multiplying from the right by $e_2$
gives $e_1e_2=e_2$. Hence,  $e_1=e_2=\varphi(1)$. This completes the proof of (b).
 \end{proof}

 By saying that a Jordan homomorphism
 $\varphi:A\to B$ is the {\bf sum of a homomorphism and an antihomomorphism} we mean that 
 there exist a homomorphism
 $\varphi_1:A\to B$ and an antihomomorphism
 $\varphi_2:A\to B$ such that
\begin{equation}
    \label{inthes}
\varphi=\varphi_1+\varphi_2\quad\mbox{and}\quad \varphi_1(A)\varphi_2(A)= \varphi_2(A)\varphi_1(A)=(0).\end{equation}
 The second condition  guarantees that $\varphi_1+\varphi_2$ is a Jordan homomorphism.

 A particularly nice situation when this occurs is when $B$ is unital and contains a central idempotent $e$ such that $x\mapsto e\varphi(x)$ is a homomorphism and $x\mapsto (1-e)\varphi(x)$ is  an antihomomorphism (so we take $\varphi_1(x)= e\varphi(x)$ and
 $\varphi_2(x)= (1-e)\varphi(x)$).
 We remark that the image of $\varphi$ is not necessarily an associative algebra in this case (consider, for example, $\varphi:M_n(F)\to M_n(F)\oplus M_n(F)$, $\varphi(a)=a+a^t$). Assuming that it is associative, and, moreover, that
 $\varphi$
 is a Jordan isomorphism, 
more can be said.

 \begin{proposition}\label{l3}
Let $\varphi:A\to B$ be a Jordan isomorphism between unital associative algebras. If $B$ contains a central idempotent $e$ such that $x\mapsto e\varphi(x)$ is a homomorphism and $x\mapsto (1-e)\varphi(x)$ is  an antihomomorphism, then $f=\varphi^{-1}(e)$ is a central idempotent in $A$, the restriction of $\varphi$ to
$fA$ is an isomorphism from $fA$ onto $eB$, and the restriction of $\varphi$ to
$(1-f)A$ is an antiisomorphism from $(1-f)A$ onto $(1-e)B$.
\end{proposition}

\begin{proof}
    Proposition \ref{l2}\,(a) shows that
    $f$ is a central idempotent in $A$ and that $\varphi(fx) = e\varphi(x)$, $x\in A$. This readily implies the proposition.
\end{proof}

Sums of homomorphisms and antihomomorphisms are considered as standard Jordan homomorphisms between associative algebras. Analogously, standard Jordan homomorphisms on $H(A,\,*\,)$ are those that can be {\bf extended to  homomorphisms}---if not on the whole  $A$, then at least on $\langle H(A,\,*\,)\rangle$ (in practice, however,  
 we often have $\langle H(A,\,*\,)\rangle=A$).

It may strike the reader as surprising that we did not also mention extensions to antihomomorphisms.
However, if $\Phi:A\to B$ is an  antihomomorphism, then $x\mapsto \Phi(x^*)$ is a homomorphism that coincides with $\Phi$ on $H(A,*)$. 
Therefore, if $\varphi:H(A,\,*\,)\to B$ can be extended to the sum of a homomorphism and an antihomomorphism  (in the sense of \eqref{inthes}), then $\varphi$ can also be extended to a homomorphism.

\section{Jordan homomorphisms on $A^{(+)}$: The early development}
\label{s5}
Jordan homomorphisms were first considered (under the name ``semi-automorphisms'') by the Spanish mathematician Germán Ancochea. 
His seminal paper \cite{An1} from 1942 considered them on quaternion algebras. The subsequent paper \cite{An2} was devoted to  the following theorem.

\begin{theorem}
   \label{ThH} A  Jordan automorphism of a finite-dimensional simple algebra $A$ is either an automorphism or an antiautomorphism. 
\end{theorem}

In 1949, Kaplansky \cite{Kap} extended Theorem \ref{ThH} to the case of characteristic $2$ (see Remark \ref{rem1}), and also considered 
finite-dimensional semisimple algebras. 
In the same year, Hua in \cite{Hua}  proved the following theorem.

\begin{theorem}
   \label{ThHua} A Jordan homomorphism from a division algebra $D$ to itself  is either a homomorphism or an antihomomorphism. 
\end{theorem}

The above 
results were substantially improved in the 1950 paper \cite{JR}
 by Jacobson and Rickart   (who also coined the name ``Jordan homomorphism''). We state two of their results. The first one generalizes Hua's theorem to domains, i.e., algebras without zero-divisors.

\begin{theorem}
   \label{Thjr} A Jordan homomorphism from any algebra $A$ to a domain $B$ is either a homomorphism or an antihomomorphism. 
\end{theorem}

 The second theorem considers matrix algebras over arbitrary unital algebras. Note that, unlike the preceding theorem, it imposes conditions on $A$ rather than on $B$.

\begin{theorem}
   \label{Thjr2} 
   Let $S$ be a unital associative algebra, let $A=M_n(S)$ with $n\ge 2$, and let $B$ be any associative algebra. For every Jordan homomorphism
   $\varphi:A\to B$  there exists a   central idempotent $e\in \langle \varphi(A)\rangle$ such that $x\mapsto e\varphi(x)$
is a homomorphism and $x\mapsto (1-e)\varphi(x)$ is an antihomomorphism. 
\end{theorem}

Kadison's paper \cite{K} on isometries and Jordan $*$-isomorphisms  
also contains the following theorem. 

\begin{theorem}
Let $A$ be a von Neumann algebra and $B$ be a  
unital $C^*$-algebra.
If $\varphi:A\to B$ is a Jordan
$*$-isomorphism, then 
 there exists a  selfadjoint central idempotent $e\in B$ such that $x\mapsto e\varphi(x)$
is a homomorphism and $x\mapsto (1-e)\varphi(x)$ is an antihomomorphism.
\end{theorem}

 To the best of our knowledge, the precise structure of a Jordan $*$-isomorphism (and hence of an isometry)  from a general $C^*$-algebra $A$  onto another 
$C^*$-algebra $B$ is still unknown (what  is known is that exactly the same result as in the case where $A$ is a von Neumann algebra  does not hold \cite[Example 5.1]{B2}, although a similar but not as definitive does  \cite{St}). 

Our final theorem in this section was 
established  by Herstein in 1956 \cite{Her}. Recall  that an algebra $B$ is prime if the product of any two nonzero ideals in $B$ is nonzero.

\begin{theorem}
   \label{ThHer} A surjective Jordan homomorphism from any algebra $A$ onto a prime algebra $B$ is either a homomorphism or an antihomomorphism. 
\end{theorem}

Herstein actually required the additional assumption that the characteristic is not $3$, which was  later removed by Smiley \cite{Sm}.

We conclude the section by sketches of proofs of Theorems \ref{Thjr} and \ref{ThHer}.

\smallskip 

\noindent
{\em Sketch of proof of Theorem \ref{Thjr}.} For any $a,b\in A$ write 
$$h(a,b)=\varphi(ab)-\varphi(a)\varphi(b)\quad\mbox{and}\quad k(a,b)=\varphi(ab)-\varphi(b)\varphi(a).$$
Compute $(ab)^2 + ab^2a = ab(ab)+(ab)ba$ in two different ways (by using  Proposition \ref{l1})   to obtain
\begin{equation}
    \label{eq1}
    h(a,b)k(a,b) =0.
\end{equation} 
Since $B$ is a domain, one of the factors is $0$. 
Thus, for each $a\in A$, the union of the additive subgroups
$$G_a=\{b\in A\,|\, h(a,b)=0\}\quad\mbox{and}\quad
H_a=\{b\in A\,|\, k(a,b)=0\}$$
 is the whole $A$.  As a group cannot be the union of two proper subgroups, we have $G_a=A$
 or $H_a=A$. This means that $A$ is the union of $$G=\{a\in A\,|\, 
 G_a=A\}\quad\mbox{and}\quad 
H=\{a\in A\,|\, 
 H_a=A\}.$$ Since $G$ and $H$ are also additive subgroups, it follows that $G=A$ or $H=A$.
\hfill\(\Box\)

\smallskip 

The proof we sketched is essentially the same as in the original paper by Jacobson and Rickart. In the  following sketch of proof of Herstein's theorem, we will also use  ideas from  \cite{B1} and \cite{Sm}.

\smallskip
\noindent
{\em Sketch of proof of Theorem \ref{ThHer}.} 
We keep the notation from the preceding proof.
Computing $a(bxb)a + b(axa)b = (ab)x(ba) + (ba)x(ab)$ in two different ways, we obtain 
\begin{equation}
    \label{eq2}
     h(a,b)\varphi(x)   k(a,b) + k(a,b)\varphi(x)   h(a,b)  =0
\end{equation} 
for every $x\in A$. Fix $a,b$ and set $h=h(a,b)$, $k=k(a,b)$. Since 
$\varphi$ is surjective, we have
$hyk + kyh=0$ for every $y\in B$. Hence,
$h(ykz)k= - kykzh$, $y,z\in B$. On the other hand, $(hyk)zk= - kyhzk = kykzh$.
Comparing, we obtain $hykzk=0$, $y,z\in B$. Since $B$ is prime, this gives $h=0$ or $k=0$. We have thus arrived at the same situation as in the preceding proof. \hfill\(\Box\)

\smallskip 

Note that 
\eqref{eq1} and \eqref{eq2} hold for any Jordan homomorphism between any associative algebras. These two equations have other applications.

\section{Jordan homomorphisms on $H(A,\,*\,)$: The early development}\label{s6}

The pioneering results on Jordan homomorphisms on 
$H(A,\,*\,)$ were obtained by Jacobson and Rickart in their 1952 paper \cite{JR2}. In particular, they proved the following.

\begin{theorem}
       \label{ThjrH} 
   Let $S$ be a unital associative algebra, let $A=M_n(S)$ with $n\ge 3$, and let $B$ be any associative algebra.
   Suppose $A$ is endowed with an involution $*$ such that the matrix units $e_{ii}$, $i=1,\dots,n$, belong to 
$H(A,\,*\,)$. Then every Jordan homomorphism $\varphi:H(A,\,*\,)\to B$ can be extended to a homomorphism.
\end{theorem}

Note that, unlike in Theorem \ref{Thjr2},  $n$ is assumed to be different from $2$. Example \ref{exampleJJ} below shows that Theorem   \ref{ThjrH} does not hold for $n=2$.

An idempotent $e$ in an algebra $A$ is called a {\bf full idempotent} if $AeA=A$, i.e., the ideal generated by $e$ is the whole $A$. Observe that the matrix units $e_{ii}$ are full idempotents.
As they are pairwise orthogonal (i.e., $e_{ii}e_{jj}=0$ if $i\ne j$), their sums $e_{i_1 i_1} + \dots + e_{i_k i_k}$ are also full idempotents.

In \cite{Mar}, Martindale generalized Theorem \ref{ThjrH} as follows.

\begin{theorem}
       \label{ThMart} Let $A$ be a unital  associative algebra with involution $*$ and let $B$ be any associative algebra.
If $H(A,\,*\,)$ contains pairwise orthogonal full idempotents $e_1,e_2,e_3$ whose sum is $1$, then 
 every Jordan homomorphism $\varphi:H(A,\,*\,)\to B$ can be extended to a homomorphism.
\end{theorem}

Martindale also examined the situation where there are only two  idempotents, but then an extra condition is needed.   

Somewhat later, Jacobson \cite{Jac} proved the characteristic-free  version of Martindale's theorem and considered it from a conceptual viewpoint, based on the following notion.

Let $J$ be a Jordan algebra. An associative algebra $U$ together with  a Jordan algebra homomorphism  
$u:J\to U^{(+)}$  
is called a {\bf special universal envelope of $J$}  if 
 for every Jordan algebra homomorphism $\varphi:J\to B^{(+)}$, where $B$ is an associative algebra, there exists a unique associative algebra homomorphism $\chi: U\to B$  such that the diagram
$$\xymatrixcolsep{5pc}\xymatrixrowsep{3pc}\xymatrix{J  \ar@{->}[r]^u \ar@{>}[dr]^{\varphi} & U \ar@{>}[d]^{\chi}
  \\
 & B
}
$$
is commutative. It is easy to see that a special universal envelope of $J$ exists and is unique up to isomorphism.

Theorem \ref{ThMart} can be reformulated as that $A$ is the
special universal
envelope of  the Jordan algebra $H(A,\,*\,)$.

Special universal envelopes play an important role in the theory of Jordan algebras. As witnessed by the recent papers \cite{Bez, BZ}, they  can also be used as  a  tool in the problem of  describing Jordan homomorphisms.

We finally remark that Jacobson also observed in \cite{Jac} that Martindale's theorem can be extended to {\bf quadratic Jordan algebras}, as defined by McCrimmon in \cite{Mcqua}.

\section{Examples of nonstandard Jordan homomorphisms}\label{s7}

In this section, we gather together several examples of Jordan homomorphisms $\varphi:J\to B$ that are not standard; in the case where $J=A$ is an associative algebra this means that $\varphi$ is not the sum of a homomorphism and an antihomomorphism, and in the case where $J=H(A,\,*\,)$ this means that $\varphi$ cannot be extended to a  homomorphism. 

We will provide only basic information. The details can be found in the original sources.

\begin{example1} \label{nonex}
Recall that  the 
Grassmann algebra  in $n$ generators, which we denote by $G_n$, is an $2^n$-dimensional associative algebra with generators   $e_1,\dots,e_n$ that satisfy the relations $e_i^2=e_ie_j+e_je_i=0$ for all $i,j$. It is quite easy to find nonstandard examples of Jordan homomorphisms on $G_n$, see \cite[Example 2.2]{B2}, \cite[Examples 3.9 and 3.10]{B3}, and \cite[Example 2.13]{BZ}. We state only \cite[Example 3.10]{B3} as we will refer to it later: The linear map $\varphi:G_3\to G_3$ given 
by  \begin{align*} & \varphi(1)=1,\,\,\, \varphi(e_1)=e_2e_3,\,\,\,\varphi(e_2)=e_1e_3,\,\,\,\varphi(e_3)=e_1e_2,\\
     &\varphi(e_1e_2)=e_3,\,\,\, \varphi(e_1e_3)=e_2,\,\,\,\varphi(e_2e_3)=e_1,\,\,\,\varphi(e_1e_2e_3)=e_1e_2e_3,\end{align*}
       is  a Jordan automorphism that is not the sum of a homomorphism and an antihomomorphism.
\end{example1}

\begin{example1}\label{exampleJJ}
An example showing that Theorem   \ref{ThjrH} does not hold for $n=2$ was first obtained in \cite{JJ}. We will present a version from \cite{Mar}.

Let $\mathbb H$ denote the algebra of quaternions and let $A=M_2(\mathbb H)$. Endow $A$  with the involution $*$ given by
$\left[\begin{smallmatrix} a_{11} & a_{12}\cr a_{21} & a_{22}
\cr
\end{smallmatrix} \right]^* = \left[\begin{smallmatrix} \overline{a}_{11} & \overline{a}_{21}\cr \overline{a}_{12} & \overline{a}_{22}
\cr
\end{smallmatrix} \right],$ where
$\overline{a}$ stands for the conjugate of $a$ in $\mathbb H$.
The Jordan algebra $H(A,\,*\,)$ is $6$-dimensional with standard basis
\[\mbox{$b_1=\left[\begin{smallmatrix} 1 & 0\cr 0 & 0
\cr
\end{smallmatrix} \right]$, 
$b_2=\left[\begin{smallmatrix} 0 & 0\cr 0 & 1
\cr
\end{smallmatrix} \right]$, $b_3=\left[\begin{smallmatrix} 0 & 1\cr 1 & 0
\cr
\end{smallmatrix} \right]$, $b_4=\left[\begin{smallmatrix} 0 & i\cr -i & 0
\cr
\end{smallmatrix} \right]$, 
$b_5=\left[\begin{smallmatrix} 0 & j\cr -j & 0
\cr
\end{smallmatrix} \right]$, $b_6=\left[\begin{smallmatrix} 0 & k\cr -k & 0
\cr
\end{smallmatrix} \right]$.}\] 
The linear map from $H(A,\,*\,)$ to itself that swaps $b_4$ and $b_5$ and fixes all the other $b_i$  
is a  Jordan automorphism  that 
 cannot be extended to a homomorphism.
\end{example1}

\begin{example1} \label{ex73}
Let $A$ be the algebra of all upper triangular matrices over $F$ 
and let $B$ be the subalgebra of 
$M_n(F)\oplus M_n(F)$ consisting of all  $a+b$ where $a$ is upper triangular, $b$ is lower diagonal, and $a$ and $b$ have the same diagonal. Then  $$\varphi:A\to B,\quad \varphi(a)=a+a^t,$$ is a Jordan homomorphism that is not the sum of a homomorphism and an antihomomorphism \cite[Remark 2.1]{Ben} (although it obviously is such a sum if we replace $B$ by the larger algebra consisting of all  $a+b$ with $a$ upper triangular and $b$ lower triangular \cite[Remark 2.3]{Ben}).
\end{example1}

  \begin{example1}\label{mexa0}
Let $A$ be a unital associative algebra satisfying the 
 following two conditions:
\begin{itemize}
    \item If $e$ is a central idempotent in $A$ such that
   $[A,A]\subseteq eA$, then $e=1$,
    \item $A$ contains a proper ideal $I$ 
    and a subalgebra $C$   contained in the center $Z(A)$ 
    such that  $1\in C$ and $A=I\oplus C$  (the vector space direct sum).   
\end{itemize}For instance, $A$ can be the algebra  obtained by adjoining a unity to 
any noncommutative algebra $I$ without  nonzero central idempotents.

Let  $\psi:A\to I$ and
$\pi:A\to C$  be the projections.  Note that  $B=A\times I^{\rm op}$ becomes a unital associative algebra under the componentwise  linear structure and  multiplication given by
\begin{equation*}\label{lab}
    (x,u^{\rm op})(y,v^{\rm op}) = (xy, \pi(x)v^{\rm op}+\pi(y)u^{\rm op} +u^{\rm op}v^{\rm op}).
\end{equation*}The map
$$\varphi:A\to B,\quad \varphi(x)= (x,\psi(x)),$$
is a Jordan homomorphism that is not 
the sum of a homomorphism and an antihomomorphism \cite[Example 2.8]{B3}. 

Using a similar idea, 
one constructs 
nonstandard examples of  Jordan automorphisms of certain semiprime algebras, i.e.,
algebras without nonzero nilpotent ideals 
\cite[Example 2.9]{B3}. One such example, which the authors attributed to Kaplansky, appeared already in  \cite{BaxM}.
\end{example1}

\begin{example1}\label{exjsust} Recall that a linear map $d$ from an algebra $A$ to itself is called a {\bf derivation} if 
$$d(xy)=d(x)y+xd(y)$$ for all $x,y\in A$.

    Let
 $d$ and $g$ be derivations of a unital associative algebra $A$, let 
$J$ be an ideal of $A$, and let   $\pi:A\to A/J$ be the canonical homomorphism.
Then  $B=A\times A/J$ becomes an associative algebra under  componentwise linear structure and multiplication given by
    $$(x,u)\star (y,v) = (xy, \pi(d(x)g(y))+ \pi(x)v+u\pi(y) ).$$
Assuming that  
\[[g(x),d(x)]\in J\quad\mbox{for all $x\in A$,}\]
it follows that 
$$\varphi:A\to B,\quad \varphi(x)= \bigl(x,\frac{1}{2}\pi((dg)(x))\bigr)$$
is a Jordan homomorphism.
Moreover, if there 
exist  $a,b\in A$  such that 
$$g(a)d(b)-d(a)g(b)\notin J,$$
then $\varphi$ is not the sum of a homomorphism and an antihomomorphism.  This is the content of    \cite[Proposition 2.11]{BZ}. It was inspired by Example 1 from the Jacobson-Rickart paper \cite{JR}. 

This example is, in particular, applicable to many commutative algebras $A$. 
All we need is a pair of derivations $d,g$ of $A$ that satisfy $g(a)d(b)\ne d(a)g(b)$ for some $a,b\in A$. Then
we can simply take $J=(0)$.
\end{example1}

\section{Jordan homomorphisms on commutator ideals} \label{s8}   

A glance at Examples \ref{ex73} and \ref{mexa0} shows that while the corresponding Jordan homomorphisms are not standard on the whole algebra $A$, they are standard (i.e., sums of homomorphisms and antihomomorphisms) on some large ideals of $A$. Since the algebras $A$ from these two examples are not exotic but rather 
common,  this suggests that it is natural to study  restrictions of Jordan homomorphisms to some ideals.

In an attempt to generalize Herstein's Theorem \ref{ThHer},
Baxter and Martindale showed in \cite{BaxM}  that for any  Jordan epimorphism $\varphi$ from an algebra $A$ onto a semiprime algebra $B$ there exists an essential ideal $E$ of $A$ such that the restriction of $\varphi$ to $E$ is the sum of a homomorphism and an antihomomorphism. 
This result was refined and generalized  in \cite{B1}.

As mentioned in
 Example \ref{mexa0},
Jordan isomorphisms between semiprime algebras are not always sums of homomorphisms and antihomomorphisms
on the whole algebra. The results from \cite{BaxM} and \cite{B1} are therefore in some sense optimal. However, as they concern an unspecified essential ideal, it may not be easy to use them in applications.

Another glance at Examples \ref{ex73} and \ref{mexa0} shows that the corresponding Jordan homomorphisms are sums of homomorphisms and antihomomorphisms on the {\bf commutator ideal}  $K(A)$ of $A$, i.e., the ideal generated by all commutators $[x,y]$ in $A$. Considering the action of Jordan homomorphisms $\varphi$ on $K(A)$ seems natural also because our goal is to describe the action of $\varphi$ on the associative product 
$xy$, which is equal to $x\circ y  +  \frac{1}{2}[x,y]$.
The action of $\varphi$ on $x\circ y$ is known by definition, so the essence of the problem is to understand the action  of $\varphi$
on commutators.

 The study of Jordan homomorphisms on commutator ideals was initiated  in 
\cite{B2} and then continued in the recent papers 
\cite{B3} and \cite{BZ}. We will present the main findings on this 
matter from 
 the latter two papers.

 The paper \cite{B3} introduces the following notion: A Jordan homomorphism $\varphi:A\to B$ is said to be {\bf splittable} if the 
 ideal of $\langle \varphi(A)\rangle$ (= the subalgebra of $B$ generated 
by the image of $\varphi$) generated by all elements $\varphi(xy)-\varphi(x)\varphi(y)$,
$x,y\in A$, has trivial intersection with the ideal of $\langle \varphi(A)\rangle$ generated by all elements $\varphi(xy)-\varphi(y)\varphi(x)$,
$x,y\in A$. 

Finding non-examples is  harder than finding examples. Let us therefore mention that the Jordan homomorphism $\varphi$ from Example
\ref{nonex} is not splittable and, moreover, the restriction of $\varphi$ to $K(G_3)$ is not the sum of a homomorphism and an antihomomorphism.

The central result of \cite{B3} reads as follows.

\begin{theorem}\label{tsplit} Let $A$ and $B$ be associative algebras.
    A splittable Jordan homomorphism $\varphi:A\to B$ is
the sum of a homomorphism and an antihomomorphism on the commutator ideal $K(A)$ of $A$.
\end{theorem}

As corollaries, extensions and modifications of several known results were obtained.   The main point of Theorem  \ref{tsplit} is that it enables a unified approach to different problems on Jordan homomorphisms. 

That said, the technical condition that a Jordan homomorphism is splittable cannot always be easily checked. The following result from  the recent paper \cite{BZ} by the present authors  gives a more clear picture.

\begin{theorem}\label{mt1} Let $A$ and $B$ be associative algebras.
 If $A$ is unital and    $A=K(A)$, then for every Jordan epimorphism 
  $\varphi:A\to B$  there exists a central  idempotent $e$ in $B$ such that $x\mapsto e \varphi(x)$ is a homomorphism and $x\mapsto (1-e) \varphi(x)$ is an antihomomorphism.
\end{theorem}

Since $B$ is an arbitrary algebra,
the condition that 
$\varphi$ is surjective  only means that its image is an associative algebra. We do not know whether this assumption is necessary.

The condition that $A=K(A)$ is justified by Example \ref{exjsust}. Indeed, among other results derived from this example in \cite{BZ}, 
it was shown    that
if $A\ne K(A)$, then there exists a commutative unital algebra $C$, an $F$-algebra $B$, and a 
       Jordan homomorphism $\varphi:A\otimes C\to B$ that is not the sum of a homomorphism and an antihomomorphism.

    \section{The approach via tetrad-eating T-ideals} \label{s9}
Let $A$ be an associative algebra.
For any $n\in\NN$ and $a_1,\dots,a_n\in A$,  write
$$\{a_1,a_2,\dots,a_n\} = a_1a_2\cdots a_n+a_n\cdots a_2a_1.$$
These elements are called $n$-{\bf tads}.  The case where $n=4$ is of special interest. The elements
$$\{a_1,a_2, a_3, a_4\} = a_1a_2a_3a_4+a_4a_3a_2a_1$$
are called {\bf tetrads}. 

In what follows, we use the notation and terminology
introduced in Section \ref{s4}.

 Consider $n$-tads $\{x_1,x_2,\dots,x_n\}$ as elements of the free algebra $F\langle X\rangle$.
It clear that they   are Jordan polynomials for  $n\le 3$. However, the tetrad 
$\{x_1,x_2,x_3,x_4\}$
is not. Indeed, the linear span 
of $e_1,e_2,e_3,e_4$ in the Grassmann algebra $G_4$ is clearly a (special) Jordan algebra, but does not contain $\{e_1,e_2, e_3, e_4\} = 2e_1e_2e_3e_4$.

An important theorem by
Cohn \cite{Cohn} states that the Jordan subalgebra
 of 
 $F\langle X\rangle^{(+)}$ 
  generated by $X\cup \{X,X,X,X\}$ contains $n$-tads $\{x_1,x_2,\dots,x_n\}$ for every $n\in\NN$.

    Let $A,B$ be  associative algebras, let $A$ be endowed with an involution $*$, and let $\varphi:H(A,\,*\,)\to B$ be a Jordan homomorphism. It is obvious that there can only be one homomorphism
    $\Phi:\langle H(A,\,*\,)\rangle \to B$ that extends $\varphi$, defined by
    $$\Phi\left(\sum h_{1}\cdots h_{n}\right)=\sum \varphi(h_{1})\cdots \varphi(h_{n}). $$
    The problem, of course, is whether $\Phi$ is well-defined.
    
    Since $H(A,\,*\,)$ is closed under $n$-tads, an obvious necessary condition for  the well-definedness of $\Phi$ is that $\varphi$ satisfies
    $$\varphi(\{h_1,\dots,h_n\})=\{\varphi(h_1),\dots,\varphi(h_n)\}$$
    for all $n\in \NN$ and all $h_1,\dots,h_n\in H(A,\,*\,)$.
    In light of the aforementioned Cohn's theorem, it is enough to consider  the case where $n=4$.

    Although this necessary condition is not, in general, sufficient, it often turns out to  be of essential importance in showing that $\Phi$ is well-defined. 
    The question whether a Jordan homomorphism on $H(A,\,*\,)$ preserves tetrads may thus  be decisive.

Recall that an ideal  $T$ of the free algebra $F\langle X\rangle$ is called a {\bf T-ideal} if
$\varphi(T)\subseteq T$ or every algebra endomorphism $\varphi$ of 
$F\langle X\rangle$. For any associative algebra $A$, we write $T(A)$ for the ideal of $A$ consisting of values of elements from $T$ on $A$.

For example, the commutator ideal $K=K(F\langle X\rangle)$ of $F\langle X\rangle$ is a T-ideal and the set of  values of elements from $K$ on $A$ is exactly the commutator ideal $K(A)$ of $A$.

Let $J\langle X\rangle$ be the free Jordan $F$-algebra on the set $X$ (see 
\cite{Jacobson, McC, fourauthors}). We can also talk about T-ideals $T$ of $J\langle X\rangle$ and ideals $T(J)$ of Jordan algebras $J$.

In \cite{Z}, the second author constructed a T-ideal $T$ of  $J\langle X\rangle$ having the following two properties:
\begin{enumerate}
    \item[(a)]
    $\{T(SJ\langle X\rangle),SJ\langle X\rangle,SJ\langle X\rangle,SJ\langle X\rangle\}\subseteq SJ\langle X\rangle$.

    \item[(b)] If $J$ is a finite-dimensional simple Jordan algebra that is not an algebra of a symmetric bilinear form (see Example \ref{ex3}), then $T(J)=J$.
\end{enumerate}

   We call (a)  the {\bf tetrad-eating property} of $T$. 

   There are actually many  T-ideals having these two properties.  For reader's convenience, we define explicitly the one used by McCrimmonn in \cite{Mc}, which is
 the 
    first paper  systematically exploring the applicability of such T-ideals in the study of Jordan homomorphisms.

 We start by introducing the Jordan polynomial
$$P(x,y,z,w)=[[\bigl([[x,y],[[x,y],z]]\bigr)^2,[[x,y],w]],[[x,y],w]].$$
Let $S$ be the T-ideal of $SJ\langle X\rangle$ generated by 
$$Q(x_1,\dots,x_{12})=[[P(x_1,x_2,x_3,x_4), P(x_5,x_6,x_7,x_8)], P(x_9,x_{10},x_{11},x_{12})].$$
For a special Jordan algebra $J$, $S(J)$
is then the ideal
of  $J$ generated by all values 
$Q(a_1,\dots,a_{12})$ with $a_i\in J$.

The following theorem follows from McCrimmonn's main result.

\begin{theorem}\label{tmc}
    Let $A$ be a unital associative algebra with involution $*$.
    If  $S(H(A,\,*\,))=H(A,\,*\,)$, then every Jordan homomorphism $\varphi$ from $H(A,\,*\,)$ to an associative algebra $B$ can be extended to a homomorphism.
\end{theorem}

We remark that Martindale's Theorem \ref{ThMart}
can be deduced from Theorem \ref{tmc}. 

Versions of Theorem \ref{tmc} also hold for other special Jordan algebras, including $A^{(+)}$. However, this result does not cover Theorem \ref{mt1}
 since 
$K(A)=A$ does not imply $S(A)=A$ (for example, $K(M_2(F))=M_2(F)$ while $S(M_2(F))=(0)$).

 The next step in using the method of tetrad-eating T-ideals was made by
 Martindale  \cite{Martindale} who generalized one 
 of  McCrimmonn's results.

Finally, in our recent paper \cite{BZ} we extended Theorem \ref{tmc}
by removing the main assumption that $S(H(A,\,*\,))=H(A,\,*\,)$, as well as the assumption that $A$ is unital. The conclusion, however, concerns only the ideal of $A$ associated with the corresponding T-ideal.
To formulate this result, we need a definition.

Let $\varphi:H(A,\,*\,)\to B$ be a Jordan homomorphism. Write $B'=\langle \varphi(H(A,\,*\,))\rangle$ and  
$${\rm Ann}(B')=\{b\in B'\,|\, bB'=B'b=(0)\}.$$ 
Clearly, ${\rm Ann}(B')$ is an ideal of $B'$. 
Composing $\varphi$ with the canonical homomorphism 
$B'\to B'/{\rm Ann}(B')$ we obtain a Jordan homomorphism
$\overline{\varphi}: 
H(A,\,*\,)\to B'/{\rm Ann}(B')$.
 We say that $\varphi$ is {\bf annihilator-by-standard} if $\overline{\varphi}$ is standard, i.e., $\overline{\varphi}$ can be extended to a homomorphism.

The theorem from \cite{BZ} reads as follows.

 \begin{theorem}\label{te}
    There exists a ${\rm T}$-ideal $G$ of $ F\langle X\rangle$ with the following properties:
    \begin{enumerate}
        \item If $A$ is an associative algebra with involution $*$ and $B$ is any associative algebra, then the restriction of any  Jordan homomorphism    $\varphi:H(A,\,*\,)\to B$ to
    $H(G(A),\,*\,)$ is annihilator-by-standard.
    \item 
 Every element of the ${\rm T}$-ideal of identities of the algebra of $2\times 2$ matrices is nilpotent modulo $G$.
    \end{enumerate}
\end{theorem}

An analogous result holds for Jordan homomorphisms 
from $A$ to $B$. Of course, this is not surprising in light of Remark \ref{raa}.  What is more interesting is the following proposition from \cite{BZ}, which shows that ``annihilator-by-standard'' cannot be replaced by ``standard''.

\begin{proposition}\label{pez}
      \label{pe2} 
      There exist an associative algebra $A$, an associative algebra $B$ with involution $*$,
      and a Jordan isomorphism $\varphi:A\to H(B,\,*\,)$ such that
      for an arbitrary nonzero {\rm T}-ideal $P$, the restriction of $\varphi$ to $P(A)$ is not standard (i.e., is not the sum of a homomorphism and an antihomomorphism).
\end{proposition}

\section{The approach via functional identities} 
\label{s10}

A {\bf functional identity}  can be informally described as an identical relation on a subset of a ring which involves arbitrary  functions  that are considered as unknowns. While functional identities are formally more general than polynomial identities, in practice the theory of functional identities
serves as a complement, rather than a generalization, of the theory of polynomial identities. 

The study of functional identities was initiated by the first author in the early 1990's. In the ensuing years, several other mathematicians---notably K.\ I.\ Beidar, M.\ A.\ Chebotar, and W.\ S. Martindale 3rd---contributed to the development of the  theory. Its fundamentals are presented in detail in the 2007 book \cite{FIbook}, and surveyed in the  recent and  up-to-date paper \cite{Bfis}.

Functional identities have turned out to be  applicable to a variety of 
problems. The description of Jordan homomorphisms is one of them. Following the exposition from \cite[Section 6.4]{FIbook}, we will state two results obtained by Beidar, Chebotar, Martindale and the first author in \cite{BBCM}.

These statements involve the notion of a $d$-free set, so we start by its definition.

Let $R$ be a nonempty subset of
a unital ring $Q$ with center $C$ (the reason behind this notation
is that in the classical situation $Q$ is the (left) maximal ring of quotients of a prime ring $R$ and $C$ is the extended centroid of $R$). 
For any positive integer $k$,  let $R^k$ denote the Cartesian product of $k$ copies of $R$.  We also set $R^0 = \{0\}$.
 Let $m$ be a positive integer, let
 $x_1,\dots, x_m \in R$, let $1\le i\le m$, and, if $m >1$, let $1\le i < j \le m$. 
Write
\begin{eqnarray*}
\ov{x}_m &=& (x_1,\ldots,x_m)\in{R}^m,\\
\ov{x}_{m}^{i}&=& (x_1,\ldots,x_{i-1},x_{i+1},\ldots,x_m)\in{R}^{m-1},\\
\ov{x}_{m}^{ij}= \ov{x}_{m}^{ji} &=& (x_1,\ldots,x_{i-1},x_{i+1},\ldots,
x_{j-1},x_{j+1}\ldots,x_m)\in{R}^{m-2}.
\end{eqnarray*}

Let $I,J\subseteq \{1,2,\ldots,m\}$ and, for each $i\in I$ and $j\in J$,
let  
$$E_i:R^{m-1}\to Q\quad\mbox{and}\quad F_j:R^{m-1}\to Q$$
be arbitrary functions. The following are  the two fundamental functional identities:
\begin{equation}
\sum_{i\in I}E_i(\ov{x}_m^i)x_i+ \sum_{j\in J}x_jF_j(\ov{x}_m^j)
 = 
0\quad\mbox{for all $\ov{x}_m\in R^m$}\label{3S1}\end{equation}
and 
\begin{equation}
\sum_{i\in I}E_i(\ov{x}_m^i)x_i+ \sum_{j\in J}x_jF_j(\ov{x}_m^j)
 \in  C\quad\mbox{for all $\ov{x}_m\in R^m$.}\label{3S2}
\end{equation}
Obviously,
 (\ref{3S1})  implies (\ref{3S2}). One should
 therefore consider  (\ref{3S1}) and  (\ref{3S2}) as independent functional identities (i.e., the functions $E_i, F_j$ from  (\ref{3S1}) are different from equally denoted functions in  (\ref{3S2})).

 Suppose $E_i,F_j$ are of the form  
\begin{eqnarray}\label{3S3}
E_i(\ov{x}_m^i) & = & \sum_{j\in J,\atop
j\not=i}x_jp_{ij}(\ov{x}_m^{ij})
+\lambda_i(\ov{x}_m^i),\quad i\in I,\nonumber\\
F_j(\ov{x}_m^j) & = & -\sum_{i\in I,\atop
i\not=j}p_{ij}(\ov{x}_m^{ij})x_i
-\lambda_j(\ov{x}_m^j),\quad j\in J,\\
&   & \lambda_k=0\quad\mbox{if}\quad k\not\in I\cap J\nonumber
\end{eqnarray}
 for some functions
\begin{eqnarray*}
&  & p_{ij}:R^{m-2}\to Q,\;\;
i\in I,\; j\in J,\; i\not=j,\nonumber\\
&  & \lambda_k:R^{m-1}\to C,\;\; k\in I\cup J.
\end{eqnarray*}
It is straightforward to check that then (\ref{3S1}) (and hence also
(\ref{3S2})) is  fulfilled.
We call \eqref{3S3} a {\bf standard solution} of the functional identities (\ref{3S1}) and (\ref{3S2}). 

The definition 
of a $d$-free set considers the situation where the standard solutions  \eqref{3S3} are also the only solutions. More precisely,  we say that 
the set $R$ is  a {\bf
$d$-free} subset of $Q$, where $d\in\NN$,
if  
the following two conditions hold for all $m\ge 1$ and all
 $I,J\subseteq \{1,2,\ldots,m\}$:
\begin{enumerate}
\item[(a)]If
$\max\{|I|,|J|\}\le d$, then (\ref{3S1}) implies (\ref{3S3}).
\item[(b)] If
$\max\{|I|,|J|\}\le d-1$, then (\ref{3S2}) implies (\ref{3S3}).
\end{enumerate}

We can now state the two results from \cite{BBCM} (see also \cite[Theorems 6.23 and 6.26]{FIbook}). 

\begin{theorem} Let $A$ and $B$ be associative algebras with $B$ unital.
If a Jordan homomorphism $\varphi:A\to B$
is such that its image is a $4$-free subset of $B$, then there exists a central  idempotent $e$ in $B$ such that $x\mapsto e \varphi(x)$ is a homomorphism and $x\mapsto (1-e) \varphi(x)$ is an antihomomorphism.
\end{theorem}

\begin{theorem} 
Let $A$ and $B$ be associative algebras with $B$ unital and $A$ endowed with an involution $*$.
 If  $\varphi:H(A,*)\to B$ is a Jordan homomorphism such that its image is a $7$-free subset of $B$, then $\varphi$ can be
extended  to a homomorphism.
\end{theorem}

These theorems may seem vacuous without providing examples of $d$-free sets. There are many; we list some standard ones.

\begin{example1} Let $S$ be a unital ring. Then the matrix ring $M_n(S)$ is a $d$-free subset of itself provided that $n\ge d$ \cite[Corollary 2.22]{FIbook}.
\end{example1}

\begin{example1}
    A simple unital ring $A$
    is a $d$-free subset of itself
    if and only if the dimension of $A$ over its center $Z(A)$ is at least $d^2$ \cite[Corollary 2.21]{FIbook}. In particular, 
  if $A$ is infinite-dimensional over $Z(A)$, then $A$ is a $d$-free subset of itself for every $d\in\NN$.
\end{example1}

\begin{example1}
    Let $A$ be a prime ring of characteristic not $2$ (this latter assumption is not needed in (a) below) and let $Q$ be its maximal left ring of quotients (see \cite[Appendix A]{FIbook}). For $a\in A$, write $\deg(a)=n$
if $a$ is algebraic of degree $n$ over the extended centroid of $R$ (= the center of $Q$) and $\deg(a)=\infty$ if $a$ is not algebraic over $C$. Set $\deg(A)=\sup\{\deg(a)\,|\,a\in A\}$. We remark that the condition that $\deg(A)=n <\infty$ is equivalent to the condition that $A$ satisfies the standard polynomial identity of degree $2n$ but does not satisfy a polynomial identity od degree less than $2n$. Further, the condition that $\deg(A)=\infty$ 
is equivalent to the condition that $A$ is not a PI-ring (i.e., does not satisfy a polynomial identity). 

Let $R\subseteq A$. According to  \cite[Section 5.2]{FIbook}, $R$ is a $d$-free subset of $Q$ if any of the following conditions holds:
\begin{enumerate}
    \item[(a)] $R$ is a nonzero ideal of $A$ and $\deg(A)\ge d$;
       \item[(b)] $R$ is a noncentral Lie ideal of $A$ and $\deg(A)\ge d+1$;
   \item[(c)] $A$ is endowed with an involution $*$,  $R=\{x\in A\,|\, x^*=x\}$ is the set of symmetric elements, and $\deg(A)\ge 2d+1$;

      \item[(d)] $A$ is endowed with an involution $*$,  $R=\{x\in A\,|\, x^*=-x\}$ is the set of skew-symmetric elements, and $\deg(A)\ge 2d+1$;

          \item[(e)] $A$ is endowed with an involution $*$, $R$ is
          a noncentral Lie ideal of the Lie ring of skew-symmetric elements 
          $\{x\in A\,|\, x^*=-x\}$,  and $\deg(A)\ge 2d+3$.
\end{enumerate}
 Note that these results have a particularly simple interpretation if $A$ is not a PI-ring. 
\end{example1}

It is important to add that if 
$R$ is a $d$-free subset of $Q$, then so is every set $S$
such that $R\subseteq S\subseteq Q$ \cite[Corollary 3.5]{FIbook}.

\section{Jordan derivations}\label{s11} 

Let $A$ be an associative $F$-algebra. A linear map $\delta:A\to A$ is called a {\bf Jordan derivation} if
$$\delta(x^2)=\delta(x)x+x\delta(x)$$
for every $x\in A$. Since we are assuming that char$(F)\ne 2$, this condition is equivalent to
$$\delta(x\circ y)=\delta(x)\circ y + x\circ \delta(y)$$
for all $x,y\in A$.

The definition of a Jordan derivation obviously makes sense for maps defined on special Jordan algebras. However, in this short section we will restrict ourselves to Jordan derivations defined on associative algebras. Also, we will not consider the more general notion of a Jordan derivation from an algebra $A$ to an  $A$-bimodule, but only Jordan derivations from $A$ to itself.

 The standard question is whether every Jordan derivation $\delta$ of $A$ is a  derivation. Incidentally, we can also define an {\bf antiderivation} as a linear map $\delta$ of $A$ that satisfies $\delta(xy)=\delta(y)x + y \delta(x)$ for all $x,y\in A$.
Just like derivations, antiderivations are also Jordan derivations.
However, evaluating an antiderivation $\delta$ on $(x^2)y=x(xy)$ in two different ways gives $[x,y]\delta(x)=0$, which indicates that nonzero antiderivations on noncommutative algebras rarely exist. Derivations are thus indeed the only natural examples of Jordan derivations.

The study of Jordan derivations is parallel to that of Jordan homomorphisms. In fact, it is somewhat easier.  Proofs of results on Jordan homomorphisms can be usually modified to obtain analogous results on Jordan derivations.  Moreover, 
since every Jordan derivation $\delta:A\to A$ gives rise to 
the Jordan homomorphism $\varphi:A\to M_2(A)$ defined by
$\varphi(x)=\left[\begin{smallmatrix} x & \delta(x)\cr 0 & x
\cr
\end{smallmatrix} \right]$, results on Jordan  derivations can be  sometimes  directly derived from results on Jordan homomorphisms.

It is therefore more challenging to establish  results on Jordan derivations   that have no known analogs for  Jordan homomorphisms. We will state two such theorems.

In \cite{Herd}, Herstein proved that Jordan derivations of prime algebras are derivations. This  is analogous to his description of Jordan homomorphisms onto prime algebras (Theorem \ref{ThHer}). As we pointed out in Section \ref{s8}, there is no easy way to extend this latter result to semiprime algebras. Herstein's result on Jordan derivations, however, is still true in these more general algebras, i.e., the following theorem holds.

\begin{theorem}\label{Tjd}
    A Jordan derivation of a semiprime algebra is a derivation.
\end{theorem}

Theorem \ref{Tjd} was first proved by Cusack \cite{Cus} (and later reproved and extended by the first author, see \cite{B1} and references therein).

In Sections \ref{s8} and \ref{s9},
we mentioned several results stating that Jordan homomorphisms can be expressed by homomorphisms and antihomomorphisms when restricted to certain ideals. We  now state a theorem of this kind for Jordan derivations, which is simpler and more definitive. It is a special case of the main result of \cite{Bjd}.

\begin{theorem}
   Let $A$ be an associative algebra and let $U$ be the ideal of $A$ generated by
   $\Bigl[\bigl[[A,A],[A,A]\bigr],\bigl[[A,A],[A,A]\bigr]\Bigr]$. Every Jordan derivation $\delta:A\to A$ satisfies $\delta(ux)=\delta(u)x+u\delta(x)$ for all $u\in U$ and $x\in A$.
\end{theorem}


\section{Lie homomorphisms}\label{s12}

 In this last section, we touch on a topic that is similar and related to Jordan homomorphisms.

Let $A$ and $B$ be associative algebras. A linear map $\varphi:A\to B$ is called a {\bf Lie homomorphism} if 
$$\varphi([x,y])=[\varphi(x),\varphi(y)]$$
for all $x,y\in A$. Every homomorphism is a Lie homomorphism, and so is the negative of an antihomomorphism. Another example, but of a different sort, is every linear map $\tau$ from $A$ to the center of $B$ that vanishes on commutators; moreover,
the sum of $\tau$ and any Lie homomorphism $\varphi$ is again a  Lie homomorphism. 

Basic examples of Lie homomorphisms are 
thus more complex than 
basic examples of Jordan homomorphisms.
 This is one of the reasons 
why the study of Lie homomorphisms is usually more demanding. Another reason is that the Jordan product $\circ$
is  closer to the original associative product $\cdot$ than the Lie product $[\,.\,,\,.\,]$. In particular, $\circ$  coincides with $\cdot$  on pairs of commuting elements.

One  also considers Lie homomorphisms that are defined 
on various Lie subalgebras (i.e., linear subspaces closed under the Lie product) of associative algebras, such as Lie ideals,  skew-symmetric elements in algebras with involution, and their Lie ideals.
In \cite{Hcon}, Herstein 
stated various conjectures concerning the structure of Lie homomorphisms on all  these kinds of Lie subalgebras. After quite many years, they were resolved by using the theory of functional identities. The main results and their proofs are given in the book \cite{FIbook}. We also mention the recent paper \cite{BKZ} that establishes 
a 
different
type result  in which the assumptions are similar
 to those in Theorem \ref{ThMart}.

It is not our intention to discuss
the structure of  Lie homomorphisms in detail. What we would like to point out is their connection to Jordan homomorphisms. To this end, we first briefly discuss a certain connection between Jordan and Lie algebras.

In \cite{Tits}, Tits made the following observation. Let $L$ be a Lie algebra containing a subalgebra 
sl$_2$ with a basis
$e,f,h$ such that $[e,f]=h$, $[h,e]=2e$, $[h,f]=-2f$. Suppose that the operator ad$(h):L\to L$, $x\mapsto [x,h]$, is diagonalizable and has eigenvalues $-2, 0, 2$, so
$L=L_{-2}\oplus L_0\oplus L_2$ is a direct sum of eigenspaces. The eigenspace
$L_2$ with the new operation
$a\circ b=[[a,f],b]$ is then a Jordan algebra. 

Conversely, every unital Jordan algebra can be obtained in this way. This was shown by
 Tits, Kantor \cite{Kantor}, and Koecher \cite{Koe}.  The corresponding Lie algebra $L$ is not unique. For a given unital Jordan algebra $J$ there exists
a universal Lie algebra with the above property. It is called the {\bf Tits-Kantor-Koecher construction} of $J$ and is denoted by $TKK(J)$. The center $Z$ of $TKK(J)$ is contained in the zero component $TKK(J)_0$. The factor algebra $\overline{TKK(J)}=TKK(J)/Z$ is called the {\bf reduced  Tits-Kantor-Koecher construction} of $J$.

A Jordan algebra homomorphism $\varphi:J\to J'$ gives rise to a Lie algebra homomorphism from
$TKK(J)$ to $TKK(J')$. If $\varphi$ is surjective, then it also gives rise to
a homomorphism from
$\overline{TKK(J)}$ to $\overline{TKK(J')}$.

Let now 
 $J=A^{(+)}$ where $A$ is an associative algebra. It is easy to show that 
$$\overline{TKK(A^{(+)})}\cong [M_2(A),M_2(A)]/Z_A,$$
where $Z_A=[M_2(A),M_2(A)]\cap Z(M_2(A))$ and $Z(M_2(A))$ is the center of the algebra $M_2(A)$.

A Jordan epimorphism  $\varphi:A\to B$ between associative algebras  therefore induces a Lie epimorphism
between $[M_2(A),M_2(A)]/Z_A$
and $[M_2(B),M_2(B)]/Z_B$. In particular, nonstandard Jordan isomorphisms can be used to construct nonstandard Lie isomorphisms (see \cite{BKZ}).

\end{document}